\documentclass[11pt,reqno]{amsart}
\usepackage{amsaddr}
\usepackage[top=30truemm,bottom=30truemm,left=25truemm,right=25truemm]{geometry}
\usepackage[english]{babel}

\usepackage[dvips]{graphicx}
\usepackage{color}
\usepackage{amsthm}
\usepackage{amsmath}
\usepackage{amscd}
\usepackage{comment}

\newtheorem{theorem}{Theorem}[section]

\newtheorem{lemma}[theorem]{Lemma}
\newtheorem{proposition}[theorem]{Proposition}

\usepackage{amssymb}\theoremstyle{definition}

\newtheorem{remark}[theorem]{Remark}
\newtheorem{problem}{Problem}

\newtheorem{maintheorem}{Theorem}

\makeatletter
\@namedef{subjclassname@2010}{
 \textup{2010} Mathematics Subject Classification}
\makeatother

\begin{document}
  \title[Asymptotically stable control problems]{Asymptotically stable control problems by infinite horizon optimal control with negative discounting}

  \author[F. Nakamura]{Fumihiko Nakamura}
  \address[F. Nakamura]{Kitami Institute of Technology, Kitami, 090-8507, Japan}
  \email[F.Nakamura]{nfumihiko@mail.kitami-it.ac.jp}

    \date{}


\begin{abstract}
In the paper, for the system which possesses both an attractor and a stable fixed point, we first formulate new stable control problems to find the asymptotically stable control function which realizes to transit a state moving around the attractor to the stable fixed point. Then by using the ordinary differential equation based on the infinite horizon optimal control model with negative discounts, we give one of answers for the stable control problem in a two-dimensional case. Furthermore, under some conditions, we verify that the phase space can be separated to some open connected components depending on the asymptotic behavior of the orbit starting from the initial point in their components. This classification of initial points suggests that it is enable to robustly achieve a stable control. Moreover, we illustrate some numerical results for the stable control obtained by applying our focused system for the Bonhoeffer-van der Pol model.
\end{abstract}

  \maketitle

\section{Problems and main theorems}

The asymptotic stability in infinite horizon optimal control problems are well discussed in previous study. For instance, the paper \cite{Brock} raised a global stable problem for modified Hamiltonian dynamical systems and gave a sufficient condition for the existence of asymptotically stable solutions. In \cite{Haurie}, an overtaking and $G$-supported properties which guarantee the stable control for a class of non-convex systems are introduced. Also \cite{Gaitsgory} provided some conditions to obtain stable control for nonlinear systems and illustrated numerical examples of its control.

The Bonhoeffer-van der Pol (BvP) model or FitzHugh-Nagumo model are well known example on the neuroscience, which have an unique stable fixed point and stable limit cycle for appropriate parameters \cite{Barnes}. In \cite{Forger, Chang}, the method how the state can transit between the stable fixed point and the limit cycle is discussed and some numerical results are displayed. Their idea is based on variational principles and they simulate by using the method called first-order gradient algorithm \cite{Bryson}. Since the transit problems in \cite{Forger, Chang} are beyond the problem supposed in \cite{Brock}, their study inspires a generalization of the problem, that is, the problem is whether it is possible to transit the stable state such as a periodic orbit to another stable state by the additive type of perturbations. From such background, in this paper, we first formulate the generalized problem as follows.

Let $F:\mathbb{R}^n\to\mathbb{R}^n$ be a smooth function and consider an autonomous ordinary differential equation $\dot{z}=F(z)$ for $z\in\mathbb{R}^n$. Let the origin $0\in\mathbb{R}^n$ be a stable fixed point of the system, and assume that there is an attractor $A$ whose basin $B(A)$ does not contain the origin.

\begin{problem}\label{prob1}
Find a control function $\chi(t)\in\mathbb{R}^n$ such that
$$
{\rm(i)\ }\displaystyle\lim_{t\to\infty}|\chi(t)|=0\quad\quad {\rm and}
\quad \quad \ {\rm(ii)\ }U(\chi)\cap A \neq\emptyset.
$$
where $\displaystyle U(\chi):=\{z_0\in\mathbb{R}^n\ | \lim_{t\to\infty}\varphi^t(z_0)=0\}$ and $\varphi^t(z_0)$ denotes the solution of $\dot{z}=F(z)+\chi$ with $z(0)=z_0$. Furthermore, find properties of $U(\chi)$. 
\end{problem}
The properties of $U(\chi)$ are, for instance, whether the Lebesgue measure of $U(\chi)$ is large enough, small or zero. If the Lebesgue measure of $U(\chi)$ is small, then the orbit may not be able to reach to zero by small perturbations. 

Now, if the control function $\chi(t)$ is determined by some differential equation $\dot{\chi}(t)=G(z,\chi)$, the problem \ref{prob1} can be reformulated as the following initial value problem.

\begin{problem}\label{prob2}
Find a smooth function $G:\mathbb{R}^{2n}\to\mathbb{R}^n$ and a set $U\subset\mathbb{R}^{2n}$ such that 
$$
{\rm(i)\ }\lim_{t\to\infty}\Phi^t(z_0,\chi_0)=0
\quad {\rm for}\quad (z_0,\chi_0)\in U,
\quad\quad {\rm and}\quad \quad
{\rm(ii)\ } \pi_1(U) \cap A \neq\emptyset.
$$
where $\Phi^t(z_0,\chi_0)$ denotes the solution of $\dot{z}=F(z)+\tilde{\chi}$ and $\dot{\chi}=G(z,\chi)$ with $z(0)=z_0$, $\chi(0)=\chi_0$, $\tilde{\chi}=(\chi_1,\cdots,\chi_k,0,\cdots,0)$ for some $k\leq n$, and $\pi_1(z,\chi)=z$ is a projection for $z$. Furthermore, find properties of $U$. 
\end{problem} 

Here, the reason why we use $\tilde{\chi}$ is because it is desired that we use less control factors from control point of view.

In this paper, we succeeded to give the partial answer to the problems in the case $n=2$. More precisely, consider the two-dimensional system $\dot{z}=F(z)$ which has only two attractors, stable fixed point 0 and stable limit cycle $\gamma^s$ such that $0\in I(\gamma^s)$ and their basins satisfy ${\rm cl}(B(\gamma^s)\cup B(0))=\mathbb{R}^2$, where $I(\gamma)$ is a set of inside points of the closed curve $\gamma$ and ${\rm cl}(X)$ implies a closure of the set $X$. In this case, there is also a unique unstable limit cycle $\gamma^u$ between 0 and $\gamma^s$. Then, we focus on the next autonomous ordinary differential equation:
\begin{eqnarray}
(\ast)\ \ \begin{cases}
\dot{z}=F(z)+\tilde{q}\\
\dot{q}=-[\rho I+D_F^T(z)]q
\end{cases}\ \ \ \text{with}\ \ \  q=(q_1,q_2),\ \tilde{q}=(q_1,0).\nonumber
\end{eqnarray}
where $\rho$ is a positive constant, $D_F(z)$ is the Jacobian matrix of $F$ at $z$, $I$ is the $n\times n$ identity matrix and $X^T$ is transpose of $X$. Assume that the solution $\Phi^t(z,q)$ of the system ($\ast$) exists for any initial point $(z,q)\in\mathbb{R}^4$ and $t\in\mathbb{R}$. Denoting $z=(x,y)$ and $F=(f,g)$, we can write it as the following differential equation:
\begin{eqnarray}
\begin{cases}
\dot{x}=f(x,y)+q_1\\
\dot{y}=g(x,y)\\
\dot{q_1}=-(\rho+f_x(x,y))q_1-g_x(x,y)q_2\\
\dot{q_2}=-f_y(x,y)q_1-(\rho+g_y(x,y))q_2
\end{cases}\nonumber
\end{eqnarray}
The next theorem shows that, by using the system ($\ast$), we can obtain the control function $\chi=(q_1,0)$ which satisfies (i) and (ii) in Problem \ref{prob1} under the assumption (A1) introduced later.
\begin{maintheorem}\label{maintheorem}
For the system ($\ast$) with the assumptions (A1), there exists a set $U\subset\mathbb{R}^4$ such that 
$$
{\rm(i)\ }\lim_{t\to\infty}\Phi^t(z_0,q_0)=0
\quad {\rm for}\quad (z_0,q_0)\in U,
\quad\quad {\rm and}\quad \quad
{\rm(ii)\ } \pi_1(U) \cap \gamma^s\neq\emptyset.
$$
\end{maintheorem}

Note that we can immediately find $\Phi^t(z_0,q_0)$ decays exponentially by adjusting $\rho$ satisfying (A1).
In addition to (A1), if the conditions (A2)-(A5) and (AA) are assumed, we can show the next theorem which tells us that the set $U$ of initial points exists as an open connected subset in $\mathbb{R}^{4}$.

\begin{maintheorem}\label{maintheorem2}
For the system ($\ast$) with the assumptions (A1)-(A5) and (AA), there exist four disjoint open connected subset $A_1,A_2,A_3,A_4\subset\mathbb{R}^4$ such that  
${\rm cl}(A_1\cup A_2\cup A_3\cup A_4)=\mathbb{R}^4$ and
\begin{eqnarray}
\lim_{t\to\infty}\Phi^t(z_0,q_0)=
\begin{cases}
(0,0) & \text{for } (z_0,q_0)\in A_2,\\
\gamma^s\times\{0\} & \text{for } (z_0,q_0)\in A_3,\\
\infty & \text{for } (z_0,q_0)\in A_1\cup A_4.
\end{cases}\nonumber
\end{eqnarray}
\end{maintheorem}

The optimal solution for the control is generally calculated by using the variational principle which is a general method to find functions which extremize the value of some giving functional. One of the most famous theorem in such optimal problems is Pontryagin's Maximum Principle \cite{Aseev2,Tauchnitz} which tell us necessary or sufficient conditions for optimal solutions. Moreover, the principle for infinite horizon are well-discussed in \cite{Haurie, Oliveira}.  

It is known that generally the equations derived from variational principles are act on a space with double dimensions, because it consists of original state variables and control variables. However, the original stable fixed point becomes saddle in the space for the additive type of perturbations. This implies that, in the case that the target state is the stable fixed point, the control variables may diverge even if the state can reach to a target point. Then, our problem is to find the optimal solution which realizes that not only the state can move to the target state but also the control variables converges to zero. 

To solve this issue, we consider the equations ($\ast$) from the variational principle for the functional having discount effects. There are many previous study for the model with positive discounting, for example \cite{Kamien, Carlson, Aseev1}. In the papers \cite{Brock,Haurie,Gaitsgory}, they also consider the positive discounting. On the other hand, there is less study focused on a negative discount. The book \cite{Kamien} says that the fixed point can be stable if the discount rate is negative for linear differential equations. Although the book \cite{Carlson} does not treat the negative discount case, they say that the negative discount rate gives the most importance to what happens in the distant future and this should amplified the stabilizing force. Thus, the reason why we focus on the system ($\ast$) is the fact that we expect one of answer for the Problem \ref{prob1} and \ref{prob2} can be achieved by using the system associated with the infinite horizon optimal control problem with ``negative'' discounting. We summarize the process to obtain the system ($\ast$) in Appendix \ref{appA}.

We consider the two-dimensional systems in this paper in order to apply the concrete example (BvP model). Indeed, we apply our theorem and display our numerical results in section \ref{application}. Moreover, for one-dimensional case, we can derive the same results under less conditions. Although the one-dimensional case is relatively simple, we summarize it in Appendix \ref{appB} since the argument helps us to understand the two-dimensional model.

\subsection{Notations and assumptions}

Before we mention our assumptions, we give some notations. First the Jacobian matrix of the system ($\ast$) can be calculate by
$$
J_*(z,q)=\begin{pmatrix}
f_x(z) & f_y(z) & 1 & 0 \\
g_x(z) & g_y(z) & 0 & 0 \\
-h_1(z,q) & -h_2(z,q) & -(\rho+f_x(z)) & -g_x(z) \\
-h_3(z,q) & -h_4(z,q) & -f_y(z) & -(\rho+g_y(z)) 
\end{pmatrix}
$$
where
\begin{eqnarray}
\begin{pmatrix}
h_1(z,q) & h_2(z,q)  \\
h_3(z,q) & h_4(z,q)  
\end{pmatrix}
=
\begin{pmatrix}
f_{xx}(x,y)q_1+g_{xx}(x,y)q_2 & f_{xy}(x,y)q_1+g_{xy}(x,y)q_2 \\
f_{yx}(x,y)q_1+g_{yx}(x,y)q_2 & f_{yy}(x,y)q_1+g_{yy}(x,y)q_2
\end{pmatrix}.\label{hessian}
\end{eqnarray}
Here $f_x$ or $f_{xy}$ denote the partial derivative. It is obvious that the divergence of the system ($\ast$) satisfies
${\rm div}_*=-2\rho$, which implies volume contracting as $t\to\infty$.

Moreover, we prepare the following stable and unstable sets for the fixed points $(z_*,q_*)$ and periodic orbit $\gamma$ of the system ($\ast$) as follows:
\begin{eqnarray}
W^s(z_*,q_*)&:=&\{(z,q)\in\mathbb{R}^4\ |\ \Phi^t(z,q)\to (z_*,q_*) \text{ as } t\to\infty\}\nonumber\\
W^u(z_*,q_*)&:=&\{(z,q)\in\mathbb{R}^4\ |\ \Phi^t(z,q)\to (z_*,q_*) \text{ as } t\to -\infty\}\nonumber\\
W^s(\gamma)&:=&\{(z,q)\in\mathbb{R}^4\ |\ d(\Phi^t(z,q),\gamma) \to 0 \text{ as } t\to\infty\}\nonumber\\
W^u(\gamma)&:=&\{(z,q)\in\mathbb{R}^4\ |\ d(\Phi^t(z,q),\gamma) \to 0 \text{ as } t\to -\infty\}\nonumber
\end{eqnarray}
where 
$$
d((z,q),\gamma):=\inf_{w\in\gamma}|(z,q)-(w,0)|.
$$
and $|\cdot|$ denotes the usual Euclidean norm.

Next, we define the set $D_\rho\in\mathbb{R}^2$ as 
$$
D_\rho:=\{z\in\mathbb{R}^2\ |\ \text{the matrix } -[\rho I+D_F(z)] \text{ is negative definite}\},
$$
where a $n\times n$ matrix $A$ is positive (negative) definite if the symmetric matrix $(A+A^T)/2$ is positive (negative) definite, that is, $x^T\frac{A+A^T}{2} x > 0$ ($<0$) holds for any vector $x\in\mathbb{R}^n\backslash \{0\}$. It is well-known that a symmetric matrix $A$ is positive (negative) definite if and only if all eigenvalues for $A$ are positive (negative). 

Now, the assumptions (A1)-(A5) are stated as follows:

\begin{itemize}
\item[(A1)] $I(\gamma^s)\subset D_\rho$.
\item[(A2)] $\#\{z\in\mathbb{R}^2\ |\ det[\rho I+D_F(z)]=0, g(z)=0\}=2$.
\item[(A3)] For any nontrivial fixed point $(\tilde{z},\tilde{q})$ for system ($\ast$) satisfying the equations 
\begin{eqnarray}\label{ntfp}
det[\rho I+D_F(z)]=0,\ \  [\rho I+D_F(z)]q=0\ (q\neq 0),\ \  F(z)+\tilde{q}=0,
\end{eqnarray}
the following inequality holds:
\begin{eqnarray}
&&\begin{pmatrix}
g_x(\tilde{z}) & g_y(\tilde{z})
\end{pmatrix}
\begin{pmatrix}
h_4(\tilde{z},\tilde{q}) & -h_2(\tilde{z},\tilde{q}) \\
-h_3(\tilde{z},\tilde{q}) & h_1(\tilde{z},\tilde{q})
\end{pmatrix}
\begin{pmatrix}
g_x(\tilde{z})\\ 
g_y(\tilde{z})
\end{pmatrix}\nonumber\\
&&\hspace{5cm}+\rho(h_1(\tilde{z},\tilde{q})g_y(\tilde{z})-h_2(\tilde{z},\tilde{q})g_x(\tilde{z}))>0\nonumber
\end{eqnarray}
\item[(A4)] There exist a set $K\subset\mathbb{R}^2$ and a function $V:\mathbb{R}^2\to\mathbb{R}$ such that
$$\text{
 (i) $V(z)>0$ in $\mathbb{R}^2\backslash K$, \quad
 (ii) $\dot{V}(z(t))<0$ in $\mathbb{R}^2\backslash K$,\quad
 (iii) $V(z)\to\infty$ if $|z|\to\infty$.
}$$
\item[(A5)] The matrix $-[\rho I+D_F(z)]$ is positive definite for $z\in\mathbb{R}^2\backslash K$ where the set $K$ is of (A4).
\end{itemize}
Fortunately, Theorem \ref{maintheorem} can be hold under only the assumption (A1) for our system ($\ast$). 
(A2) means that there are two nontrivial fixed point of the system ($\ast$) which satisfy the equation \eqref{ntfp}. Although this condition seems to be a strong condition, we can calculate only two fixed points for BvP model in the section \ref{application} by choosing appropriate $\rho$. (A3) plays an important role in order that the stable set $W^s(\tilde{z},\tilde{q})$ becomes a three-dimensional stable manifold in $\mathbb{R}^4$. The Lyapunov function-like assumption (A4) might be a natural condition since we consider that the original system $\dot{z}=F(z)$ has only one fixed point and stable limit cycles as its attractors. (A5) is also important for the separation of $\mathbb{R}^4$ in Theorem \ref{maintheorem2}.

Finally, one of difficulties of the arguments is that we do not know the existence of non-trivial closed orbit for our four-dimensional differential equation. The famous Poincare-Bendixson Theorem cannot be applied to more than three-dimensional system generally. Although some of previous study said about the non-existence of closed orbits, for instance \cite{Smith} showed it by the condition of sum of the first and second eigenvalues, it is difficult to apply their results to our system. Thus, to prove Theorem \ref{maintheorem2}, assume the following condition:
\begin{itemize}
\item[(AA)] The system ($\ast$) has no non-wandering point in $K\backslash D_\rho\times\mathbb{R}^2$, where $(z,q)\in\mathbb{R}^4$ is said to be non-wandering point if, for any open set $U\subset\mathbb{R}^4$ containing $(z,q)$ and any $T>0$, there exists $t>T$ such that $\Phi^t(U)\cap U\neq\emptyset$.
\end{itemize}

\section{Proof of Theorem A}

To prove the Theorem \ref{maintheorem}, we prepare the next proposition and the lemma.
\begin{proposition}\label{negativedef}
For a linear system $\dot{q}(t)=A(x)q(t)$ with $q(0)=q_0$ where $A(x)$ is a $n\times n$ matrix for $x(t)$ which is a continuous map in some compact subset $K\subset\mathbb{R}^n$ and $q(t)\in\mathbb{R}^n$ is a vector for $t\in\mathbb{R}$. If $A(x)$ is negative definite for any $x \in K$, then $|q(t)|$ converges to 0 as $t\to\infty$. 
\end{proposition}
\begin{proof}
Since $A(x)$ is negative definite, we can calculate as follows:
\begin{eqnarray}
\frac{d}{dt}\lvert q(t)\rvert^2&=&q^T(t)\dot{q}(t)+\dot{q}^T(t)q(t)\nonumber\\
&=& q^T(t)A(x)q(t)+q^T(t)A^T(x)q(t)\nonumber\\
&=& q^T(t)(A(x)+A^T(x))q(t)<0 \ \text{for any $t>0$}.\nonumber
\end{eqnarray}
Therefore, $\lvert q(t)\rvert^2$ is monotonically decreasing as $t\to\infty$, and $\lvert q(t)\rvert$ must converge to 0 as $t\to\infty$ since $\frac{d}{dt}\lvert q(t)\rvert^2<0$ for any $q(t)\neq 0$. 
\end{proof}

\begin{lemma}
The point $(0,0)\in \mathbb{R}^4$ becomes a stable fixed point of the system ($\ast$). Moreover, the set $\gamma^s\times\{0\}$ and $\gamma^u\times\{0\}$ become a stable and unstable limit cycle in $\mathbb{R}^4$, respectively.   
\end{lemma}

\begin{proof}
By calculating Jacobian matrix for the fixed point $(0,0)\in \mathbb{R}^4$, we immediately find the eigenvalues of the matrix are $\lambda_1(0)$, $\lambda_2(0)$, $-\lambda_1(0)-\rho$ and $-\lambda_2(0)-\rho$. Then $(0,0)$ becomes a stable fixed point since the assumption (A1) since a negative definite matrix $-[\rho I+D_F(0)]$ has two eigenvalues with a negative real part.

Next, from the assumption (A1), when $z(t)$ moves around the neighborhood of the periodic orbit in $D_\rho$, then $-[\rho I+D_F(z)]$ is always negative definite. Thus, by Proposition \ref{negativedef}, $q_1(t)$ and $q_2(t)$ converge to $0$ Therefore, $\gamma^s\times\{0,0\}$ and $\gamma^u\times\{0,0\}$ become a stable and unstable limit cycle on $\mathbb{R}^4$. 
\end{proof}

\subsection*{\bf Proof of Theorem \ref{maintheorem}}

By the assumption (A1), there is a small $\varepsilon>0$ such that $\displaystyle\lim_{t\to\infty}\Phi^t(B_\varepsilon(0,0))=0$ where $B_\varepsilon(0,0)$ is a open ball in $\mathbb{R}^4$ with center $(0,0)$ and radius $\varepsilon>0$. We will show that $\pi_1(\Phi^{-t}(B_\varepsilon(0,0)))\cap \gamma^s \neq \emptyset$. Assume that the set is empty set. Since $0\in I(\gamma^s)\subset \mathbb{R}^2$, it must be hold that $\pi_1(\Phi^{-t}(B_\varepsilon(0,0)))\subset I(\gamma^s)$ which implies $\pi_1(\Phi^{-t}(B_\varepsilon(0,0)))\subset D_\rho$. Then, any orbit $z(-t)$ is included in the set $D_\rho$ for any $t\geq 0$ so that the matrix $[\rho I+D_F(z)]$ is always positive definite. From the similar argument of Proposition \ref{negativedef}, we have $|q|$ monotonically increases to infinity as $t\to-\infty$ for any $(z_0,q_0)\in B_\varepsilon(0,0)$. 

Assume that $q_1(-t)$ is bounded for any $t>0$, then $|q_2|$ must increase since $|q|$ is increasing. Moreover, there is some $T$ such that $q_2(-t)$ is always positive (or negative) for any $t>T$ and goes to $\infty$ (or $-\infty$). For the case $q_2(-t)\to\infty$ (the case $q_2(-t)\to-\infty$ is same), consider the equation $\dot{q}=(\rho+f_x)q_1+g_xq_2$. Note that the minus sings of the original equation are replaced by plus sings because we now consider the inverse time direction, $t\to-\infty$. If $g_x=0$, since $\rho+f_x$ is positive due to positive definite matrix $[\rho I+D_F(z)]$, we have $q_1(t)\geq e^{(\rho+\min f_x)t}\to\infty$ as $t\to\infty$, which contradicts to the boundedness of $q_1$. If $g_x$ is always positive (or negative), since $z(t)$ and $q_1(t)$ are bounded but $q_2(t)$ increase, we have $\dot{q}_1(t)>C$ (or $\dot{q}_1(t)<-C$) for some $C>0$ and sufficient large $t$. This means that $q_1$ cannot be bounded which is contradiction. If the case that $g_x(z(t))$ has oscillated and repeated positive and negative values, the function $G(t):=g_x(t)q_2(t)$ oscillates and diverges. In this case, let $\{s_i\}$ and $\{t_i\}$ be times satisfying  $G(s_i)=G(t_i)=0$, $\dot{G}(s_i)>0$ and $\dot{G}(t_i)<0$. Then, since $z(t)$ and $q_1(t)$ are bounded but $G(t)$ is unbounded, the time interval $t_{i+1}-t_i$ must converge to zero as $i\to\infty$. Indeed, letting 
$$
a_i=\int_{s_i}^{t_i}G(t)dt \quad {\rm and } \quad b_i=\int_{t_i}^{s_{i+1}}G(t)dt,
$$
we have $b_{i+1}\leq b_i<0<a_i\leq a_{i+1}$ and $\int_{s_1}^{t_{i+1}}G(t)dt=\sum a_i + \sum b_i$ must be bounded since $q_1$ is bounded. This implies $t_{i+1}-t_i\to0$ as $i\to\infty$ due to the increasing $q_2$. 

However, if $t_{i+1}-t_i$, either $\dot{z}(t)$ is unbounded or $g_x(t)\to0$, and both lead a contradiction. Therefore, we found $q_1(t)$ is unbounded for $t\to-\infty$. 

Finally, considering the equation $\dot{x}=-f(z)-q_1$, the function $q_1(t)$ must oscillate and diverge. Indeed if $q_1$ does not oscillate and does diverge, there are $C>0$ and $T>0$ such that $\dot{x}(t)>C$ (or $\dot{x}(t)<-C$) for any $t>T$, which implies $x(t)$ is unbounded. Hence, let $\{s_i\}$ and $\{t_i\}$ be times satisfying  $q_1(s_i)=q_1(t_i)=0$, $\dot{q}_1(s_i)>0$ and $\dot{q}_1(t_i)<0$. Then, since $x(t)$ and $f(z)$ are bounded but $q_1(t)$ is unbounded, the time interval $t_{i+1}-t_i$ must converge to zero as $i\to\infty$. However, this is impossible since the imaginary part of eigenvalues of the matrix $[\rho I + D_F(z)]$ (if they are complex numbers) is bounded so that the frequency of the oscillation of $q_1$ cannot be infinity which contradicts to $t_{i+1}-t_i\to0$. (When the eigenvalues are real number, $t_{i+1}-t_i\to0$ cannot occur since $q_1$ must get away from 0.) 

From these arguments, we can show that $z(t)$ must go out in $D_\rho$ as $t\to-\infty$ for any $(z_0,q_0)\in B_\varepsilon(0,0)$, which implies $\pi_1(\Phi^{-t}(B_\varepsilon(0,0)))\cap \gamma^s \neq \emptyset$. This completes the proof.

\section{Proof of Theorem B}
The ideas of the proof of Theorem \ref{maintheorem2} are to show that firstly the stable sets  $W^s(\tilde{z},\tilde{q})$ and $W^s(\gamma^u)$ become three-dimensional smooth manifolds by the Lemma \ref{stable1} and \ref{stable2}. Secondary, by knowing behaviors of almost all orbits $\Phi^t(z_0,q_0)$ as $t\to-\infty$ from Lemma \ref{conv}, we can derive these stable manifolds separate $\mathbb{R}^4$ to four-disjoint open connected sets. To prove this, the separation theorem in \cite{Hirsch} which is well-known in differential topology are used.
Denote the two non-trivial fixed point by $(\tilde{z}_1,\tilde{q}_1)$ and $(\tilde{z}_2,\tilde{q}_2)$ satisfying the equation \eqref{ntfp}.

\begin{lemma}\label{stable1}
Under the assumptions (A2) and (A3), the Jacobian $D_*(\tilde{z}_i,\tilde{q}_i)$  has three eigenvalues with negative real parts for the non-trivial fixed points $(\tilde{z}_i,\tilde{q}_i)$, $i=1,2$, so that $W^s(\tilde{z}_i,\tilde{q}_i)$ becomes a three-dimensional smooth manifold.
\end{lemma}

\begin{proof}
For the fixed point $(\tilde{z}_i,\tilde{q}_i)$, we can calculate the characteristic polynomial as follows:
\begin{eqnarray}
\Lambda(\lambda)&:=&det[\lambda I-D_F(\tilde{z}_i,\tilde{q}_i)]\nonumber\\
&=&\lambda^4+2\rho\lambda^3-\{\rho^2+3(f_x+g_y)\rho+(f_x+g_y)^2-h_1\}\lambda^2\nonumber\\
&& -\{2\rho^2+3\rho(f_x+g_y)+(f_x+g_y)^2-h_1\}\rho\lambda -B(\tilde{z}_i)
\end{eqnarray}
where
$$
B(\tilde{z}_i)=B(\tilde{x}_i,\tilde{y}_i)=-
\begin{pmatrix}
g_x & g_y
\end{pmatrix}
\begin{pmatrix}
h_4 & -h_2 \\
-h_3 & h_1
\end{pmatrix}
\begin{pmatrix}
g_x\\ 
g_y
\end{pmatrix}
-\rho(h_1g_y-h_2g_x).
$$
Then we have
\begin{eqnarray}
\frac{d}{d\lambda}\Lambda(\lambda) &=& 4\lambda^3+6\rho\lambda^2-2\{\rho^2+3(f_x+g_y)\rho+(f_x+g_y)^2-h_1\}\lambda\nonumber\\
&&-2\{2\rho^2+3(f_x+g_y)\rho+(f_x+g_y)^2-h_1\}\nonumber\\
&=& 2(\lambda+\rho/2)(2\lambda^2+3\rho\lambda-\{2\rho^2+3(f_x+g_y)\rho+(f_x+g_y)^2-h_1\})\nonumber
\end{eqnarray}
Thus, one of the extremal values of $\Lambda(\lambda)$ takes at $\lambda=-\frac{\rho}{2}$, and we find that if $\Lambda(0)<0$, then the equation $\Lambda(\lambda)=0$ has three negative solutions and a positive solution when all eigenvalues are real numbers. The case $\Lambda(\lambda)$ has two complex eigenvalues, we can calculate its real parts are both negative if $\Lambda(0)<0$. Therefore, the stable manifold theorem \cite{Teschl} and the assumption (A3) which implies the inequality $\Lambda(0)<0$ guarantees $W^s(\tilde{z}_i,\tilde{q}_i)$ is a three-dimensional smooth manifold.

\end{proof}

\begin{lemma}\label{stable2}
$W^s(\gamma^u)\simeq S^1\times\mathbb{R}^2$ is a three-dimensional smooth manifold.
\end{lemma}

\begin{proof}
Due to $\gamma^u\subset D_\rho$, $|q(t)|$ decreases monotonically and converges to 0 as $t\to\infty$ if $z(t)$ moves around the neighbor of $\gamma^u$, which implies both $q_1(t)$ and $q_2(t)$ are stable. Thus, from the stable manifold theorem for a periodic orbit \cite{Teschl}, $W^s(\gamma^u)\simeq S^1\times\mathbb{R}^2$ becomes a three-dimensional smooth manifold.
\end{proof}

\begin{lemma}\label{lyap-prop}
If (A4) is assumed for the original two-dimensional system $\dot{z}=F(z)$, then, for any $z_0$, there exists $T>0$ such that the solution $z(t)$ with $z(0)=z_0$ belongs to $D_\rho$ for $t>T$.
\end{lemma}

\begin{proof}
Assume that the solution $z(t)$ does not belong to $D_\rho$ for any $t$. Since  $V(z(t))$ is monotonic decreasing and positive, $V(z(t))$ converges to some $a>0$, that is, $a\leq V(z)\leq V(z(0))$ holds. Letting $A$ be a set $\{z\in\mathbb{R}^2\ |\ a\leq V(z)\leq V(z(0))\}$, the set $A$ becomes clearly closed and bounded set. Thus $A$ is compact. Then, since $\dot{V}$ has a supremum in $A$, that is, $\displaystyle\sup_{z\in A}\dot{V}(z)=-b<0$, we have
$$
V(z(t))=V(z(0))+\int_0^t\dot{V}(z(t))dt\leq V(x(0))-bt.
$$
Therefore, $V(x(t))<0$ holds for $t>V(x(0))/b$ which contradicts to the assumption (i) in (A4).

\end{proof}

\begin{lemma}\label{conv}
Assume (A1)-(A5). Then $|z(t)|\to\infty$ and $|q(t)|\to 0$ as $t\to-\infty$ for any initial points $(z_0,q_0)$ in $\mathbb{R}^4$ except with $W^u(\tilde{z}_1,\tilde{q}_1)\cup W^u(\tilde{z}_2,\tilde{q}_2)\cup W^u(\gamma^u)\cup\{0,0\}\cup\gamma^s\times\{0\}$.
\end{lemma}

\begin{proof}
First we have already know the convergent of orbits from initial points in $W^u(\tilde{z}_1,\tilde{q}_1)\cup W^u(\tilde{z}_2,\tilde{q}_2)\cup W^u(\gamma^u)\cup\{0,0\}\cup\gamma^s\times\{0\}$. Thus, we chose the initial point $(z_0,q_0)$ in $\mathbb{R}^4$ except with these sets. 

From Theorem \ref{maintheorem}, there is $T>0$ such that $\pi_1(\Phi^{-t}(z_0,q_0))\in \mathbb{R}^2\backslash D_\rho$ for $t>T$. Assume that $\pi_1(\Phi^{-t}(z_0,q_0))\in K$ for any $t>T$. In the case that $|q|$ is increasing and diverge, we can lead a contradiction from the similar arguments with Theorem \ref{maintheorem}. In the case $|q|$ is bounded, $\pi_1(\Phi^{-t}(z_0,q_0))$ moves in some bounded region for any $t>T$, which implies the existence of non-wandering point $(\overline{z},\overline{q})$ in $K$ since we can take subsequence $\{t_i\}$ such that $\Phi^{-t_i}(z_0,q_0)\to(\overline{z},\overline{q})$ as $i\to\infty$ (by Bolzano-Weierstrass theorem). This contradicts to (AA). Hence, although the orbit may go back to $K$ again after it goes out from $K$, by (AA), there is $T'>T$ such that $\pi_1(\Phi^{-t}(z_0,q_0))\in\mathbb{R}^2\backslash K$ for any $t>T'$.  

Then, since $[\rho I + D_F(z)]$ is negative definite if $z\in\mathbb{R}^2\backslash K$ by (A5), we have $|q(t)|\to 0$ as $t\to-\infty$. This implies $q_1(t)\to 0$ as $t\to-\infty$ and the orbit follows the original system $\dot{z}=F(z)$. By substituting $w(t)=z(-t)$, the system becomes $\dot{w}(t)=-F(w)$ and $\dot{V}(w(t))=-\dot{V}(-z(t))>0$ in $\mathbb{R}^2\backslash D_\rho$. Thus, $V(w(t))$ monotonically increases and goes to infinity as $t\to\infty$. If $w(t)$ is bounded, $V(w)$ is also bounded, which is contradiction. Therefore, $|z(t)|\to\infty$ as $t\to -\infty$. This completes the proof.
\end{proof}

\subsection*{\bf Proof of Theorem \ref{maintheorem2}}

First prove that $W^s(\tilde{z},\tilde{q})$ separates $\mathbb{R}^4$ to two connected sets where $(\tilde{z},\tilde{q})$ is the fixed point satisfying Eq.\eqref{ntfp}. Consider $\Phi^{-t}(W^s(\tilde{z},\tilde{q})\cap B_\varepsilon(\tilde{z},\tilde{q}))$. From Lemma \ref{conv}, any orbits $\Phi^{-t}(z_0,q_0)\to\infty$ as $t\to\infty$ for $(z_0,q_0)\in W^s(\tilde{z},\tilde{q})\cap B_\varepsilon(\tilde{z},\tilde{q})$.

Now we show that the closure of $W^s(\tilde{z},\tilde{q})$ has no boundary. Assume that the closure of the boundary is not empty, namely, we have $A\subset\partial W^s(\tilde{z},\tilde{q})$. Since $W^s(\tilde{z},\tilde{q})$ is a smooth manifold, $A\nsubseteq W^s(\tilde{z},\tilde{q})$. Thus, for any $y\in A$, there exist $\{t_i\}_{i=1}^\infty$ and $(z_i,q_i)\in \partial B_\varepsilon(\tilde{z},\tilde{q})$ such that $\Phi^{-t_i}(z_i,q_i)\to y$ for some small $\varepsilon>0$. Since $\partial B_\varepsilon(\tilde{z},\tilde{q})$ is compact set, there is subsequence $\{i_k\}_k$ such that $(z_{i_k},q_{i_k})\to(\overline{z},\overline{q})\in \partial B_\varepsilon(\tilde{z},\tilde{q})$. From Lemma \ref{conv}, for any $L>0$, there is $k$ with $t_{i_k}>T$ such that $|\pi_1(\Phi^{-t_{i_k}}(z_{i_k},q_{i_k}))|\geq L$ for any $L$, which contradicts to $\Phi^{-t_{i_k}}(z_{i_k},q_{i_k})\to y\in A$ as $k\to\infty$.

Then, Theorem 4.6 in \cite{Hirsch} tells us that if $N$ is a simply connected manifold and $M\subset N$ is a connected closed submanifold of codimension 1 with $\partial M=\partial N=\emptyset$, then $M$ separates $N$. In our case, although $W^s(\tilde{z},\tilde{q})$ is not closed, we can prove that $W^s(\tilde{z},\tilde{q})$  separates $\mathbb{R}^4$ by modifying the proof of the theorem since $W^s(\tilde{z},\tilde{q})$ spreads to infinity. (Indeed, we can check similar proof does work even our situation since, for any $r>0$, the closed disc $D^4(r)\subset\mathbb{R}^4$ satisfies $\partial(W^s(\tilde{z},\tilde{q})\cap D^4(r))\subset\partial D^4(r)$. See \cite{Hirsch} for detail.)

Similarly, we have that $W^s(\gamma^u)$ also separates $\mathbb{R}^4$. Hence, there exist three separatrixes $W^s(\tilde{z}_1,\tilde{q}_1)$, $W^s(\tilde{z}_2,\tilde{q}_2)$ and $W^s(\gamma^u)$ which separate $\mathbb{R}^4$ to four connected components since each stable set has no intersection due to a uniqueness of the solution.
The four sets denote by $A_1, A_2, A_3$ and $A_4$. $W^s(\gamma^u)$ separates 0 and $\gamma^s$ since $W^u(\gamma^u)=I(\gamma^s)\times\{0\}\backslash \{0,0\}$ and $W^u(\gamma^u)$ and $W^s(\gamma^u)$ do not intersect transversally. Hence we can take the sets satisfying $0\in A_2$ and $\gamma^s\in A_3$. Moreover, the set $A_2$ and $A_3$ are bounded for $q$, that is, there is a constant $L>0$ such that $|q|<L$ for $(z,q)\in A_2\cup A_3$ since $A_2\cup A_3$ are between $W^s(\tilde{z}_1,\tilde{q}_1)$ and $W^s(\tilde{z}_2,\tilde{q}_2)$. 

Finally, there is no non-trivial attractor for our system ($\ast$) since if $z\in D_\rho$ (or $z\in K^c$), then $|q|$ has decreased (or increased), and there is no non-wandering point in $K\backslash D_\rho$. Since there is no attractor except with the stable fixed point 0 in $A_2$, the orbit $\Phi^t(z_0,q_0)$ goes to either 0 or $\infty$ for $(z_0,q_0)\in A_2$. Assuming that $\Phi^t(z_0,q_0)$ diverges, $(x(t),q(t))$ must go to $(\infty,0)$ since $|q|$ is bounded, which contradicts to (A4). Thus, $\Phi^t(z_0,q_0)$ converges to 0 for any $(z_0,q_0)\in A_2$.
Similarly, $\Phi^t(z_0,q_0)$ converges to $\gamma^s$ for $(z_0,q_0)\in A_3$. 
Since there is no attractor in $A_1$ and $A_4$, the orbit $\Phi^t(z_0,q_0)$ must diverge for $(z_0,q_0)\in A_1\cup A_4$. This completes the proof.\qed

\begin{remark}\label{final}
Even if the assumption (A4) and (A5) are not assumed, the stable set $W^s(\tilde{z},\tilde{q})$ can separate $\mathbb{R}^4$ if any orbits $\Phi^{-t}(z_0,q_0)\to\infty$ as $t\to\infty$ for $(z_0,q_0)\in W^s(\tilde{z},\tilde{q})$. However, in this case, we cannot conclude that $\Phi^t(z_0,q_0)\to 0$ (or $\gamma^s$) for any $(z_0,q_0)\in A_2$ (or $A_3$) since the sets $A_2$ and $A_3$ may be unbounded for $q$ and some orbits in the sets may diverge.

Furthermore, if not only (A4) and (A5) but also (AA) are not imposed, some non-trivial invariant set may exist. The system has no fixed point except with 0 and $(\tilde{z}_i,\tilde{q}_i)$, $i=1,2$, and non-trivial attractor for $t\to-\infty$ cannot exists because of div$_*=2\rho$. However, we cannot verify the existence of invariant set like a periodic orbit or chaos orbit. If it has two or three dimension as its unstable directions, a heteroclinic orbit between it and $(\tilde{z},\tilde{q})$ or $\gamma^u$ may exist and this makes our discussions more complicated. 
\end{remark}

\section{Application for Bonhoeffer-van der Pol model}\label{application}
In this section, we apply our results to the concrete example called Bonhoeffer-van der Pol (BvP) model described by
\begin{eqnarray}
\begin{cases}
\dot{x}=c(x-x^3/3+y-r)\\
\dot{y}=-(x+a+by)/c
\end{cases}.
\end{eqnarray}
It is well known that this system has a stable limit cycle, a unstable limit cycle and a stable fixed point when $a=0.7,b=0.8,c=3.0,r=0.342$, and the fixed point $(x_*,y_*)\fallingdotseq (0.958366,0.322957)$. We can calculate the Jacobian matrix as
$$
D_F(x,y)=\begin{pmatrix}
c(1-x^2) & c \\
-1/c & -b/c
\end{pmatrix}.
$$ 
The eigenvalues $\lambda_{\pm}$ of $D_F(x_*,y_*)$ are $- 0.011031\pm 0.966773i$. Next, by applying the system ($\ast$) for BvP model, we obtain the following autonomous four-dimensional ordinary differential equation.
\begin{eqnarray}
\begin{cases}
\dot{x}=c(x+x^3/3+y-r)+q_1\\
\dot{y}=-(x+a+by)/c\\
\dot{q_1}=-(\rho+c(1-x^2))q_1+\frac{1}{c}q_2\\
\dot{q_2}=-cq_1-(\rho-\frac{b}{c})q_2
\end{cases}
\end{eqnarray}
For example, if we give the initial point $(x_0,y_0)=(1.35,-0.26)$, $(q_1(0),q_2(0))=(0.0,-5.0)$ and $\rho=2.5$, we can observe numerically in the figure \ref{fig1} that the point near the stable limit cycle goes to the stable fixed point.

\begin{figure}[htp]
\begin{center}
\includegraphics[bb=0 0 750 350, width=15cm]{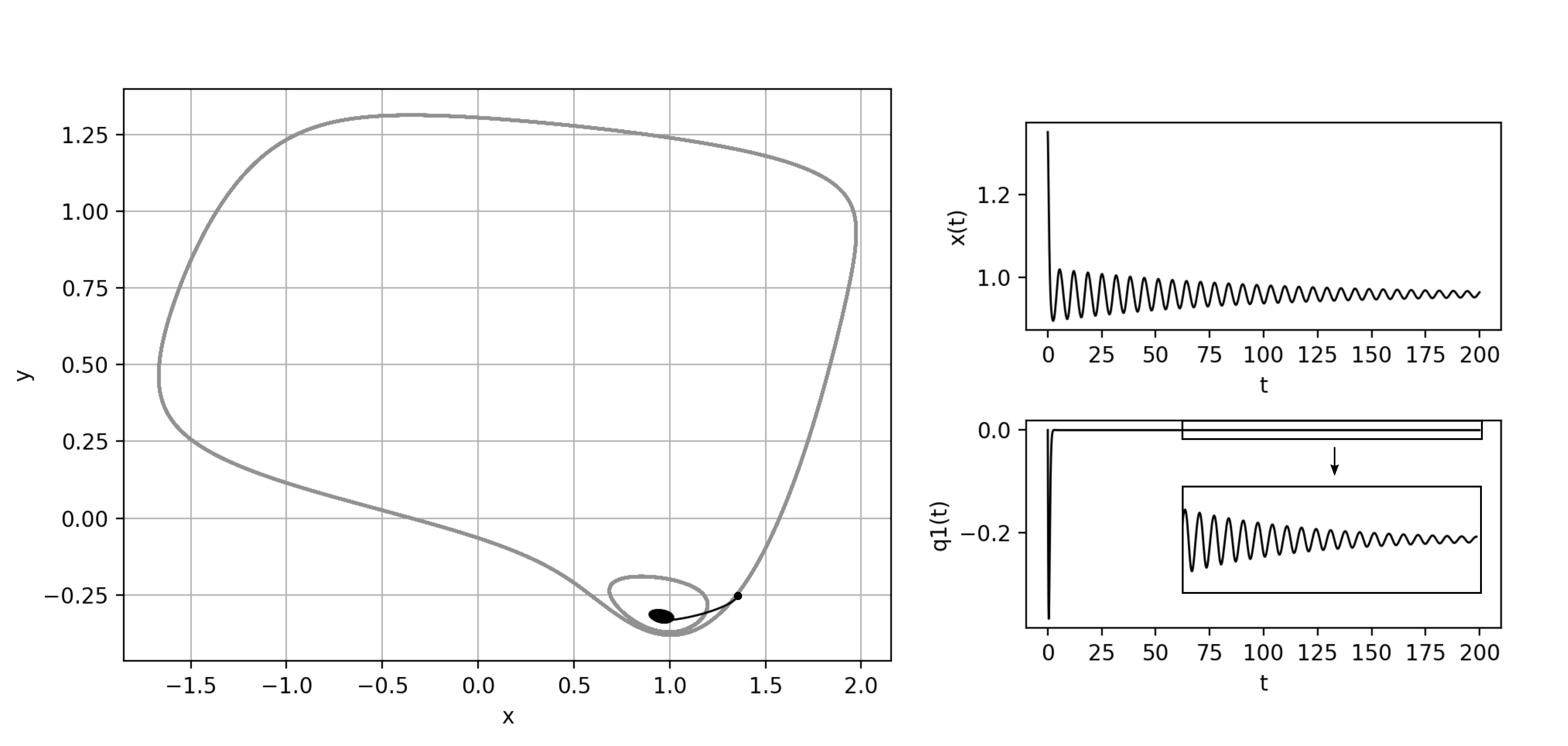}
\caption{The orbit of $(x(t),y(t))$ for the initial point $(x_0,y_0)=(1.35,-0.26)$ and $(q_1(0),q_2(0))=(0.0,-5.0)$ by our discount model with the rate $\rho=2.5$ is illustrated. The big (small) gray closed curve in the left is the stable (unstable) limit cycle for the original BvP model.  }
\label{fig1}
\end{center}
\end{figure}

Now, we can define $D_\rho$ for the model as
$$
D_\rho=\left[-\sqrt{1+\frac{\rho}{c}-\frac{(c^2-1)^2}{4c^2(c\rho-b)}},\ \sqrt{1+\frac{\rho}{c}-\frac{(c^2-1)^2}{4c^2(c\rho-b)}}\right]\times \mathbb{R}
$$
Next we estimate $\rho$ satisfying our assumption (A1)-(A5):

\begin{itemize}
\item[(A1)] Since $I(\gamma^s)\subset[-2,2]$, we must choose $\rho>9.2$. However, such $\rho$ is so large that the set of initial points becomes small and it may be difficult to capture the appropriate initial points. Moreover, Figure \ref{fig1} and \ref{fig2} suggest the assumption (A1) might be weakened.
\item[(A2)] Considering the equations $det[\rho I+D_F(z)]=(\rho+c(1-x^2))(\rho-b/c)+1=0$ and $g(x,y)=0$, we find that the number of solution is always just two.
\item[(A3)] From $\Lambda(0)=f(z)f_{xx}(z)(g_y^2(z)+\rho g_y(z)) < 0$ and $f(z)f_{xx}(z) > 0$, then (A3) is satisfied if $\rho>-g_y=1/c\fallingdotseq 0.333\cdots$.
\item[(A4)] The assumption is independent of $\rho$.
\item[(A5)] Unfortunately, for any $z\in\mathbb{R}^2\backslash D_\rho$, the matrix $-[\rho I +D_F(z)]$ cannot be positive definite so that this assumption cannot be satisfied. However, Remark \ref{final} suggest that even if (A5) is not satisfied, the space $\mathbb{R}^4$ can be separated if almost orbits diverge as $t\to-\infty$. We can check it numerically, the assumption (A5) also might be weakened.
\item[(AA)] It is difficult to show that the assumption is satisfied. However, we expect to notice the non-trivial closed orbit in the numerical experiments if it exists. 
\end{itemize}

The figure \ref{fig2} describes classifications of initial points of control variables $(q_1,q_2)$ $\in[-5,5]^2$ for the several initial points $(x_0,y_0)$ in the case $\rho=2.5$ which satisfied all assumptions. The orbit starting from the black, white and gray region goes to the stable fixed point, goes to the stable limit cycle and diverge respectively.

\begin{figure}[htp]
\begin{center}
\includegraphics[bb=0 0 800 450, width=16cm]{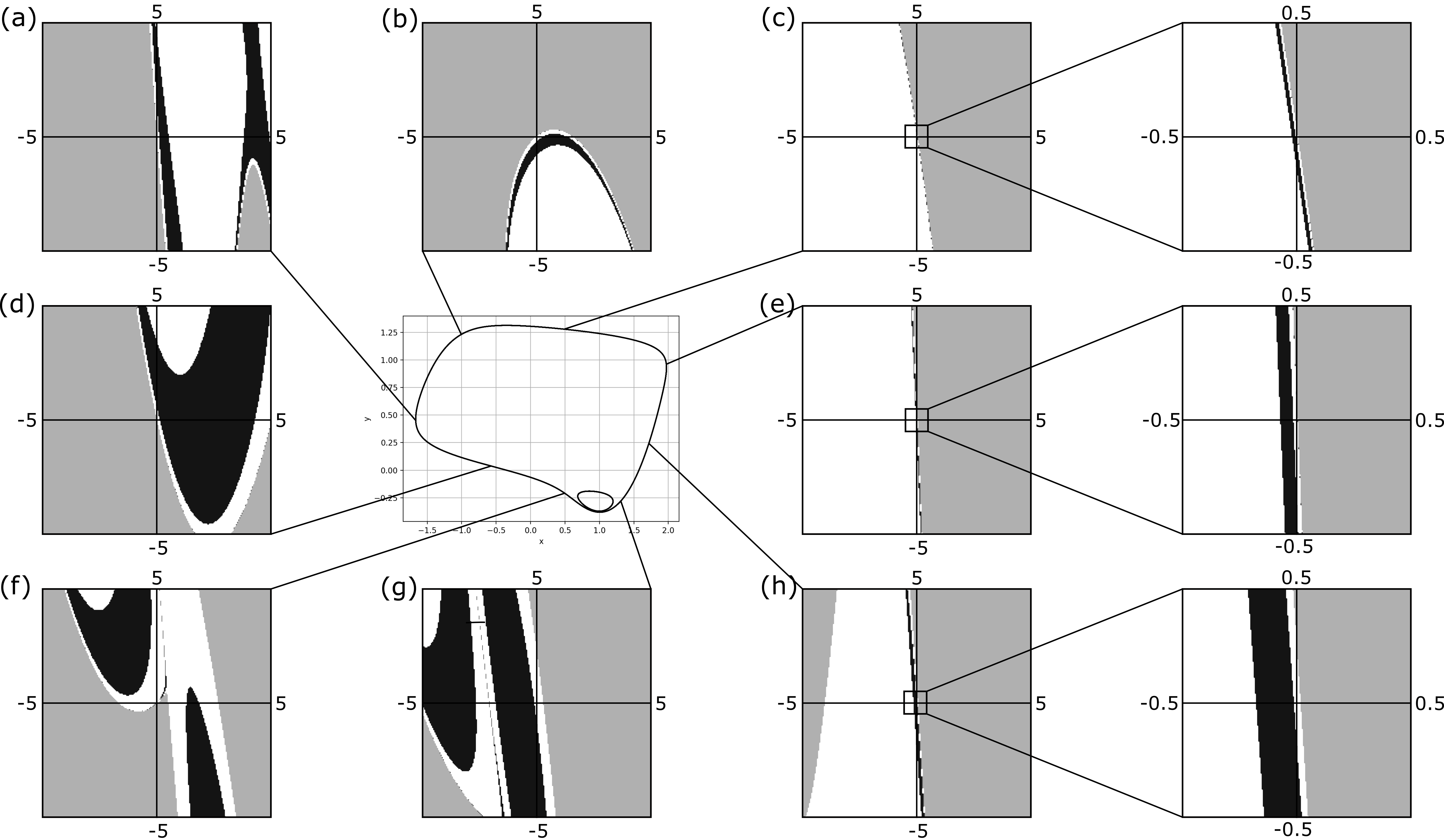}
\caption{The classifications of $(q_1$,$q_2)$ planes are illustrated for some $(x_0,y_0)$. If $(q_1(0),q_2(0))$ is chosen in the black, white and gray region, the orbit $(x(t),y(t))$ goes to the stable fixed point, goes to the stable limit cycle and diverges respectively. The initial point $(x_0,y_0)$ is (a)  ($-$1.66,0.42), (b) ($-$1.0,1.23), (c) (0.5,1.27), (d) ($-$0.62,0.04), (e) (1.98,0.93), (f) (0.5,$-$0.22), (g) (1.35,$-$0.26), (h)  (1.73,0.25).}
\label{fig2}
\end{center}
\end{figure}

\begin{remark}
In the equation ($\ast$), although we give the control variables $q_1$ to the only $x$ direction, we naturally can consider the system which has a second control variable $q_2$ into the $y$ direction as follows:
\begin{eqnarray}\label{2sideeq}
\begin{cases}
\dot{z}=F(z)+q\\
\dot{q}=-[\rho I+D_F(z)]q
\end{cases}\ \ \ \text{with}\ \ \  q=(q_1,q_2).
\end{eqnarray}
that is,
\begin{eqnarray}
\begin{cases}
\dot{x}=f(x,y)+q_1\\
\dot{y}=g(x,y)+q_2\\
\dot{q_1}=-(\rho+f_x(x,y))q_1-g_x(x,y)q_2\\
\dot{q_2}=-f_y(x,y)q_1-(\rho+g_y(x,y))q_2
\end{cases}\nonumber
\end{eqnarray}
We can also prove Theorem \ref{maintheorem} and Theorem \ref{maintheorem2} for this system if all assumptions are hold. However, it might be more difficult to satisfy the assumptions. For example, the characteristic polynomial becomes
\begin{eqnarray}\label{2sideeq2}
\Lambda(\lambda)&:=&det[\lambda I-D_F(z,q)]\nonumber\\
&=&\lambda^4+2\rho\lambda^3-\{\rho^2+3(f_x+g_y)\rho+(f_x+g_y)^2-(a+d)\}\lambda^2\nonumber\\
&& -\{2\rho^2+3\rho(f_x+g_y)+(f_x+g_y)^2-(a+d)\}\rho\lambda -\tilde{B}(z,q),
\end{eqnarray}
where
\begin{eqnarray}
\tilde{B}(z,q)&=&-\begin{pmatrix}
f_x & f_y
\end{pmatrix}
\begin{pmatrix}
h_4 & -h_2 \\
-h_3 & h_1
\end{pmatrix}\begin{pmatrix}
f_x\\ 
f_y
\end{pmatrix}
-\begin{pmatrix}
g_x & g_y
\end{pmatrix}
\begin{pmatrix}
h_4 & -h_2 \\
-h_3 & h_1
\end{pmatrix}\begin{pmatrix}
g_x\\ 
g_y
\end{pmatrix}\nonumber\\
&&-\rho(h_1g_y-h_2g_x-h_3f_y+h_4f_x)-(h_1h_4-h_2h_3),\nonumber
\end{eqnarray}
then the relation $\tilde{B}(\tilde{z},\tilde{q})<0$ must be hold instead of assumption (A3). For the BvP model, by using the equation \eqref{2sideeq2}, $\rho$ must be satisfied $\rho>c^3+1/c\fallingdotseq 27.33\cdots$. Therefore we find that using the system ($\ast$) is better than the system \eqref{2sideeq}. Note that if we consider the control variables to the only $y$ direction, we can show that $\Lambda(0)$ is always positive so that the assumption (A3) cannot be hold. This implies that the initial points cannot be classified by the separatrix $W^s(\tilde{z},\tilde{q})$.
\end{remark}

\appendix

\section{Infinite horizon with discounting}\label{appA}

Consider the infinite horizon optimal control problem described by 
\begin{eqnarray}
\begin{cases}
\text{Minimize} & \int_0^\infty |p(t)|^2/2 dt \\
\text{subject to} & \dot{z}(t)=F(z(t))+p(t),\ \ \  \displaystyle z(0)=z_0, \lim_{t\to\infty}z(t)=z_\infty
\end{cases}
\end{eqnarray}
where $F:\mathbb{R}^n\to\mathbb{R}^n$ is a smooth function, $z(t),p(t)\in\mathbb{R}^n$ are vector valued functions and $z_0,z_\infty\in\mathbb{R}^n$.
The problem is to find the optimal control function $p(t)$ to make the state $z(t)$ to move from the initial state $z_0$ to the terminal state $z_\infty$ under the differential equation $\dot{z}(t)=F(z(t))+p(t)$. Define the functional called performance function or Lagrangian by
$$
\mathcal{L}(z,p,\mu)=|p|^2/2+\mu^T(\dot{z}-F(z)-p)
$$
where the vector valued function $\mu(t)\in\mathbb{R}^n$ is as a Lagrange multiplier and $x^T$ denotes the transpose of vector $x$. By the usual variational principle, we can obtain the following system on $\mathbb{R}^{2n}$,
\begin{eqnarray}\label{sys}
\begin{cases}
\dot{z}=F(z)+p\\
\dot{p}=-D_F(z)^T p
\end{cases}
\end{eqnarray}
where $D_F(x)$ is the Jacobian matrix of $F$ for $z$ and $A^T$ denotes the transpose of the matrix $A$.

Now we will focus on the stability problem for a fixed point of the system \eqref{sys}. Assume that the system $\dot{z}=f(z)$ has a fixed point $z_*$ which have non-zero eigenvalues $\{\lambda_i\}_{i=1}^n$ of the Jacobian matrix $D_F(z_*)$. Then we immediately find that $(z_*,0)\in\mathbb{R}^{2n}$ becomes the fixed point of the system \eqref{sys} on $\mathbb{R}^{2n}$ whose eigenvalues are $\{\pm\lambda_i\}_{i=1}^n$ which implies the point $(z_*,0)$ becomes always a saddle. This suggests that the control variable $p(t)$ may extremely increase or diverge even if $z(t)$ goes to $z_*$ for sufficiently large $t$.

To avoid the trouble, we propose the optimal control problem with a discount parameter $\tilde{\rho}$ as follows:
\begin{eqnarray}
\begin{cases}\label{discountcontrol}
\text{Minimize} & \int_0^\infty e^{-\tilde{\rho}t}|p(t)|^2/2 dt \\
\text{subject to} & \dot{z}(t)=F(z(t))+p(t), z(0)=z_0, \displaystyle \lim_{t\to\infty}z(t)=z_\infty.
\end{cases}
\end{eqnarray}
Give the Lagrangian with a discount term $e^{-\tilde{\rho} t}$ by
$$
\mathcal{L}(z,p,\mu)=e^{-\tilde{\rho} t}|p|^2/2+\mu^T(\dot{z}-f(z)-p),
$$
and by the variational principle, we have
\begin{eqnarray}
\begin{cases}
\dot{z}=F(z)+p e^{\tilde{\rho} t}\\
\dot{p}=-D_F(z)^Tp.
\end{cases}\nonumber
\end{eqnarray}
Changing the variables $q(t)=p(t)e^{\tilde{\rho} t}$ leads the autonomous system on $\mathbb{R}^{2n}$,
\begin{eqnarray}\label{mainsys}
\begin{cases}
\dot{z}=F(z)+q\\
\dot{q}=\tilde{\rho} q -D_f(z)^T q.
\end{cases}
\end{eqnarray}
The eigenvalues of the Jacobian matrix for the fixed point $(z_*,0)$ for the system \eqref{mainsys} become $\lambda_i$ and $-\lambda_i+\tilde{\rho}$. Thus, when $\lambda_i<0$ for all $i$, the point $(x_*,0)$ can become stable fixed point by taking $\rho<0$ satisfying $-\lambda_i-\tilde{\rho}<0$ for all $i$. By substituting $\rho=-\tilde{\rho}$, we obtain the system ($\ast$). Therefore, we can expect that it is possible to realize the stable control by considering infinite horizon optimal control problem with ``negative'' discount.

Note that, by Legendre transformation, the system \eqref{mainsys} can be rewritten as
\begin{eqnarray}
\begin{cases}
\dot{z}=\frac{\partial H}{\partial q}\\
\dot{q}=-\frac{\partial H}{\partial z}+\tilde{\rho} q
\end{cases}
\end{eqnarray}
where the Hamiltonian $H$ is given by
\begin{eqnarray}
H(z,q,\mu):=e^{-\tilde{\rho}t}|q|^2/2+\mu^T(F(z)+q)\nonumber
\end{eqnarray}
Remark that we do not know generally whether the system is a Hamilton system if $\rho\neq 0$. In the case $\rho$ is positive, it is known that the integral in \eqref{discountcontrol} converges as $T\to\infty$ for a bounded $p(t)$ (See \cite{Carlson}). On the other hand, for the negative rate $\rho$, the integral in \eqref{discountcontrol} does not converge generally. Although the infinite horizon problem with positive discount rate is well discussed in many previous works, for example \cite{Kamien, Carlson, Aseev1}, for the negative discount rate, it has not developed enough due to the difficulty of the nonlinearity.

\section{For one dimensional control problems}\label{appB}

Our discount model applied to one-dimensional autonomous system $\dot{x}=f(x)$ can be calculated by elementary methods since we know a non-existence of closed orbit by Poincare-Bendixson theorem in two-dimensional systems. Because these calculus help us to understand arguments for the higher dimensional model, we summarise in the appendix section.

Let $f:\mathbb{R}\to\mathbb{R}$ be a smooth map and $f$ has more than two stable fixed points. We denote the sets of stable and unstable fixed points by $FP^s:=\{x_1^s,x_2^s,\cdots,x_{N+1}^s\}$ and $FP^u:=\{x_1^u,x_2^u,\cdots,x_{N}^u\}$. Consider the two-dimensional system with discount as follows:
\begin{eqnarray}\label{twodimsys}
\begin{cases}
\dot{x}=f(x)+q\\
\dot{q}=-(\rho + f'(x))q
\end{cases}
\end{eqnarray}
Let $D_\rho:=\{x\in\mathbb{R}\ |\ \rho+f'(x)>0\}$. Non-trivial fixed points of \eqref{twodimsys} can be calculated by
\begin{eqnarray}\label{ntfpdim1}
\rho +f'(x)=0,\ \ \  f(x)+q=0.
\end{eqnarray}
We assume the followings:
\begin{itemize}
\item[(A1)'] $FP^s, FP^u\in D_\rho$.
\item[(A2)'] $\#\{x\in\mathbb{R}\ |\ \rho +f'(x)=0\}=2$.
\item[(A3)'] $f(x)f''(x) > 0$ in $\mathbb{R}\backslash D_\rho$,
\item[(A4)'] $f'(x)\to-\infty$ if $|x|\to \infty$ 
\end{itemize}
For one-dimensional model, we usually give a Lyapunov function as $V(x)=-f'(x)$. Then, $V(x) > 0$ and $\dot{V}(x(t))<0$ in $\mathbb{R}\backslash$ by (A2)' and (A3)'. Therefore the assumption (A4) is satisfied by (A2)' and (A3)' together with (A4)'. Finally, the assumption (A5) is automatically satisfied by taking $K=D_\rho$ in the case of one-dimensional model. Then, we can show Theorem \ref{maintheoremdim1} under the assumption (A1)'-(A4)'.

The Jacobian matrix for the system $F$ is given by

$$
D_F(x,q)=\begin{pmatrix}
f'(x) & 1 \\
-f'(x)q & -(\rho+f'(x))
\end{pmatrix}
$$

\begin{lemma}
If $x_*\in FP^s$, then $(x_*,0)$ becomes a stable fixed point of $F$ . If $x_*\in FP^u$, then $(x_*,0)$ becomes a saddle. 
\end{lemma}

\begin{proof}
For the fixed point $(x_*,0)$, we have the Jacobian matrix is
$$
D_F(x_*,0)=\begin{pmatrix}
f'(x_*) & 1 \\
0 & -(\rho+f'(x_*))
\end{pmatrix},
$$
and the eigenvalues of the matrix are $f'(x_*)$ and $-f'(x_*)-\rho$. If $x_*\in FP^u$, then $(x_*,0)$ becomes a saddle since $-f'(x_*)-\rho<0$. If $x_*\in FP^s$, then $(x_*,0)$ becomes stable from the assumption (A1).

\end{proof}

\begin{lemma}
There are two non-trivial fixed points of $F$ satisfying \eqref{ntfpdim1}, which are saddles.
\end{lemma}

\begin{proof}
From the assumption (A2)', we have only two non-trivial fixed points of $F$. For the fixed point $(\tilde{x},\tilde{q})$, the Jacobian matrix is
$$
D_F(\hat{x},\hat{q})=\begin{pmatrix}
-\rho & 1 \\
f''(\hat{x})f(\hat{x}) & 0
\end{pmatrix},
$$
and the eigenvalues of the matrix are $(-\rho\pm \sqrt{\rho^2+4f''(\hat{x})f(\hat{x})})/2$. Therefore $(\hat{x},\hat{q})$ must be saddle from (A4). 
\end{proof}

\begin{theorem}\label{maintheoremdim1}
Assume that (A1)'-(A4)' hold. Then, there exist $(N+2)$ disjoint simply connected subset $A_0,A_1,\cdots,A_{N+1}\subset\mathbb{R}^2$ such that
${\rm cl}(A_0\cup \cdots \cup A_{N+1})=\mathbb{R}^2$ and
 $$
\lim_{t\to\infty}\Phi^t(x_0,q_0)=\begin{cases}
(x_i^s,0) & \text{ for } (x_0,q_0)\in A_i\ \ (i=1,\cdots,N)\\
\infty & \text{ for } (x_0,q_0)\in A_0 \cup A_{N+1}.
\end{cases}
$$
\end{theorem}

\begin{figure}[tp]\label{fig3}
\begin{center}
\includegraphics[bb=0 0 260 140, width=13cm]{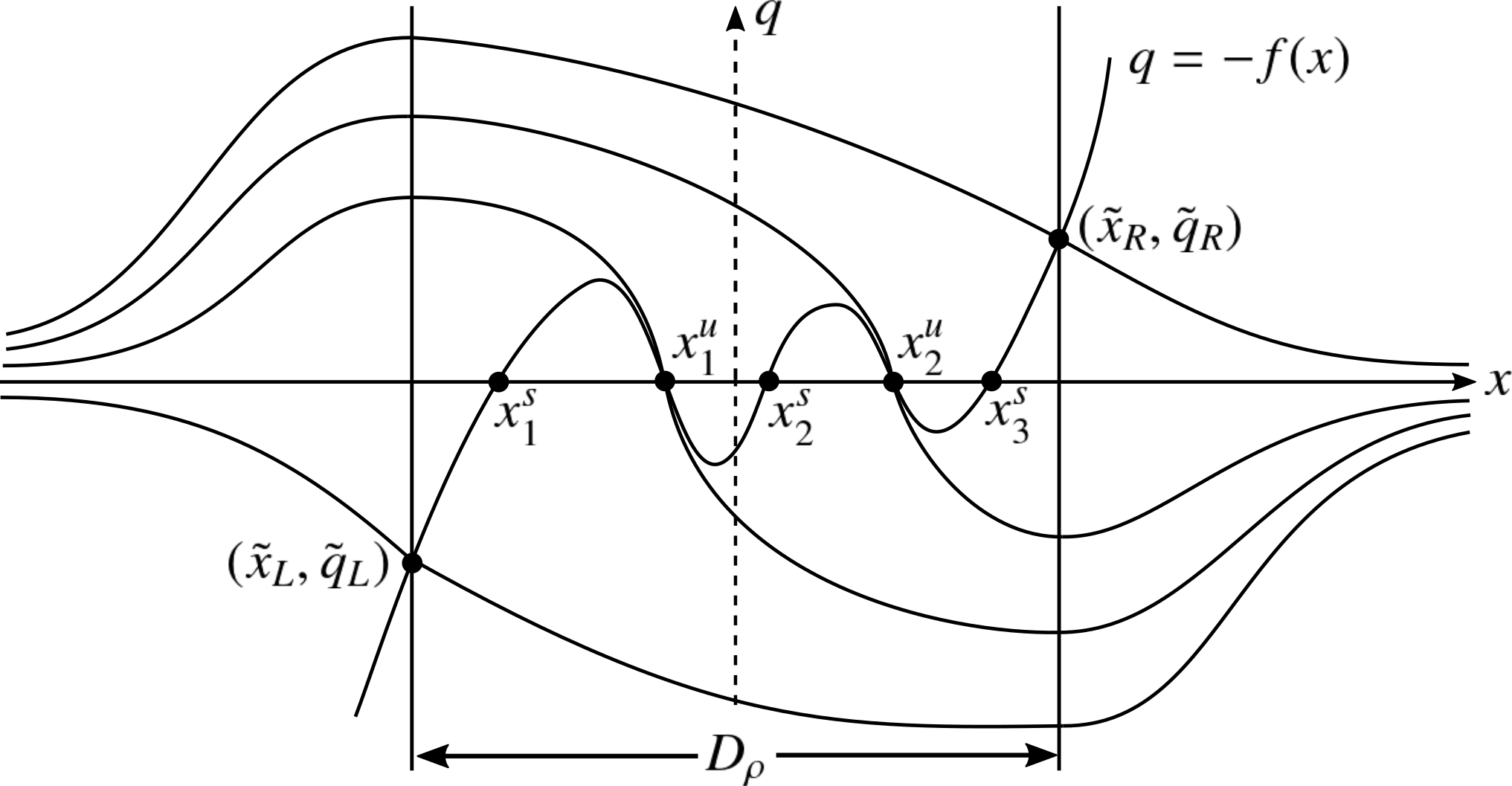}
\caption{The illustration for one-dimensional control problems. The four stable manifolds of each saddle point, $(x_1^u,0), (x_2^u,0), (\tilde{x}_L,\tilde{q}_L)$ and $(\tilde{x}_R,\tilde{q}_R)$, separate $\mathbb{R}^2$.}
\end{center}
\end{figure}

\begin{proof}

Since $D_\rho$ is an open interval set, denote $D_\rho=(x_L,x_R)$. First we show that $W^s(\tilde{x},\tilde{q})$ with $\tilde{x}=x_R$ separates $\mathbb{R}^2$ to two connected sets and is their boundary. Since $(\tilde{x},\tilde{q})$ is a saddle, the neighbor of the fixed point contains sets of initial points from which the orbit converges to $(\tilde{x},\tilde{q})$ starting. If the initial point $(x_0,q_0)$ in the neighbor satisfies $q_0+f(x_0)>0$, then $(x_0,q_0)\in D_\rho$ and $\dot{q}>0$ which mean $x(t)$ decreases and $q(t)$ increases monotonically as $t\to-\infty$ until $x(t)$ reaches $x_L$. After that, $\dot{q}$ becomes negative, so that, $q(t)\to0$ and $x(t)\to-\infty$ as $t\to-\infty$ since the line $\{(x,q):\ x<x_L, q=0\}$ becomes a part of separatrixes. If the initial point $(x_0,q_0)$ in the neighbor satisfies $q_0+f(x_0)<0$, then $q(t)\to0$ and $x(t)\to\infty$ as $t\to-\infty$ by similar arguments. Therefore, the set $W^s(\tilde{x},\tilde{q})$ separates $\mathbb{R}^2$ to two connected sets. (More rigorously, we may need the theorem 4.6 in \cite{Hirsch} or discussions of one-point compactification and Jordan-Brouwer separation theorem.)

By the same arguments, we can see that the all stable manifolds $W^s(\tilde{x},\tilde{q})$ for $\tilde{x}=x_L$ and $W^s(x_i^s,0)$ for $i=1,\cdots,N+1$ also separate $\mathbb{R}^2$. Moreover, since they have no intersection each other, the stable manifolds $W^s(\tilde{x}_1,\tilde{q}_1)$, $W^s(\tilde{x}_2,\tilde{q}_2)$ and $W^s(x_i^s,0)$ separate $\mathbb{R}^2$ to $N+2$ regions $\{A_i\}_{i=0}^{N+1}$. We can obviously name $\{A_i\}_{i=0}^{N+1}$ satisfy (i)-(iii) of the theorem.

Finally, each set $A_i$, $i=1,\cdots,N$, has only unique stable fixed points $x_i^s$ in the inside, and the orbit cannot move to other $A_j$  since all stable manifolds are separatrixes. The orbit clearly does not diverge and there is no non-trivial attractor in each $A_i$ since there is no fixed points except with $(\tilde{x}_1,\tilde{q}_1)$, $(\tilde{x}_2,\tilde{q}_2)$ and $(x_i^s,0)$. Therefore, the orbit $\Phi^y(x_0,q_0)$ with $(x_0,q_0)\in A_i$ must be converge to $(x_i^s,0)$ as $t\to\infty$. For $(x_0,q_0)\in A_0\cup A_{N+1}$, the orbit must diverge since $A_0$ and $A_{N+1}$ have no attractor in its inside. The proof is completed.

\end{proof}

\section*{Acknowledgement}
I am deeply grateful to Prof. Ichiro Tsuda (Chubu university) and Prof. Hideo Kubo (Hokkaido university) for constructive comments and warm encouragement. Moreover, I would like to thank Prof. Atsuro Sannami (Kitami institute of Technology), Prof. Okihiro Sawada (Kitami institute of Technology) and Prof. Tomoo Yokoyama (Kyoto university of Education) to give me insightful comments and suggestion. Finally, this work is supported by JST CREST Grant Number JPMJCR17A4, Japan.


\end{document}